\documentclass[10pt]{amsart}

\usepackage{amssymb,amsthm,amsmath}
\usepackage{enumerate}
\usepackage{graphicx,color}
\usepackage[hidelinks]{hyperref}

\newcommand{\dd}{\mathrm{d}}
\newcommand{\E}{\mathbb{E}}
\newcommand{\1}{\textbf{1}}
\newcommand{\R}{\mathbb{R}}
\newcommand{\C}{\mathbb{C}} 

\newcommand{\e}{\varepsilon}
\newcommand{\p}[1]{\mathbb{P}\left( #1 \right)}
\newcommand{\scal}[2]{\left\langle #1, #2 \right\rangle}
\newcommand{\red}{}

\usepackage{xpatch}
\makeatletter
\def\thm@space@setup{%
  \thm@preskip=12pt plus 0pt minus 0pt
  \thm@postskip=0pt plus 0pt minus 0pt
}
\xpatchcmd{\proof}{6\p@\@plus6\p@\relax}{\z@skip}{}{}
\makeatother

\newtheorem{theorem}{Theorem}
\newtheorem{lemma}[theorem]{Lemma}
\newtheorem{corollary}[theorem]{Corollary}

\theoremstyle{remark}
\newtheorem{remark}[theorem]{Remark}

\theoremstyle{definition}

\title{Rademacher--Gaussian tail comparison for complex coefficients and related problems}

\author{Giorgos Chasapis}

\author{Ruoyuan Liu}
\address{{\red (R.L.) School of Mathematics, The University of Edinburgh, Edinburgh, EH9 3FD, UK.}}
\email{}

\author{Tomasz Tkocz}
\address{(G. C. \& T. T) Department of Mathematical Sciences, Carnegie Mellon University; Pittsburgh, PA 15213, USA.}
\email{gchasapi@andrew.cmu.edu, ruoyuanl@alumni.cmu.edu, ttkocz@math.cmu.edu}

\thanks{TT's research supported in part by NSF grant DMS-1955175.}

\date{1st June 2021}

\begin{document}

\begin{abstract}
We provide a generalisation of Pinelis' Rademacher-Gaussian tail comparison to complex coefficients. We also establish uniform bounds on the probability that the magnitude of weighted sums of independent random vectors uniform on Euclidean spheres with matrix coefficients exceeds its second moment.
\end{abstract}

\vspace*{-3em}

\maketitle

\bigskip

\begin{footnotesize}

\end{footnotesize}

\bigskip

\section{Introduction}

Let $\e_1, \e_2, \dots$ be independent Rademacher random variables (symmetric random signs, each $\e_j$ takes the values $\pm 1$ with probability $\frac12$). Significant amount of work has been devoted to moment and tail bounds for weighted sums $S = \sum_j a_j\e_j$ in a variety of settings, with motivations and applications in areas such as statistics, or functional analysis (see, e.g. \cite{LT}).
We shall be interested in tail probabilities of the magnitude of $S$ and its higher-dimensional counterparts. 

Pinelis in \cite{P1} (see also \cite{BGH, P4}) proved the following precise deviation inequality: for every $n \geq 1$, {\red real numbers} $a_1, \dots, a_n$ and positive $t$,
\begin{equation}\label{eq:P1}
\p{\left|S \right| \geq t\sigma} \leq C\int_{t}^\infty e^{-u^2/2}\frac{\dd u}{\sqrt{2\pi}},
\end{equation}
where $S = \sum_{j=1}^n a_j\e_j$, $\sigma = (\E S^2)^{1/2} = (\sum_{j=1}^n a_j^2)^{1/2}$ and $C= \frac{2e^3}{9}$, the value of which was subsequently improved, see \cite{Bent, P2} and the optimal value established in \cite{BDz} (attained when $n=2$, $a_1 = a_2 = 1$, $t = \sqrt{2}$). An asymptotically tight bound is also known: the constant $C$ can be replaced with $1 + O(1/t)$, see \cite{P3}. Our first result provides an analogue of \eqref{eq:P1} for complex-valued coefficients $a_j$. 

Another interesting regime concerns ``typical values'' of $S$. There are universal constants $c_1, C_1 \in (0,1)$ such that for every $n \geq 1$ and {\red real numbers} $a_1, \dots, a_n$,
\begin{equation}\label{eq:typical}
c_1 \leq \p{|S| \geq \sigma} \qquad\text{and}\qquad \p{|S| > \sigma} \leq C_1
\end{equation}
The lower bound was first established in \cite{Burk}, without any explicit value of $c_1$, later with $c_1 = \frac{1}{4e^4}$ in \cite{HKwa}, with $c_1 = \frac{1}{10}$ in \cite{O} and with $c_1 = \frac{3}{16}$ in \cite{DK}. The upper bound with $C_1 = \frac{5}{8}$ was obtained in \cite{HK}. The conjecture that it holds with the sharp value $C_1 = \frac{1}{2}$ (attained again when $n=2$, $a_1 = a_2 = 1$) was attributed to Tomaszewski. Having received a lot of attention, the conjecture has recently been proved in \cite{KK} (see further references therein). Our second result provides a multidimensional extension of \eqref{eq:typical}, where the random signs $\e_j$ are replaced with uniform random vectors on the unit sphere, the coefficients $a_j$ are matrix-valued and the magnitude is measured by the Euclidean norm.

We detail our results in the next section which is followed by the section devoted to their proofs. We finish with several remarks.

\subsection*{Acknowledgments} {\red We are indebted to an anonymous referee for many valuable comments which helped significantly improve the manuscript; particularly for sharing and letting us use their slick and elegant proof of Claim 2.}

\section{Results}

\subsection{Rademacher-Gaussian tail comparison}

Here and throughout, $\scal{x}{y} = \sum_{j=1}^d x_jy_j$ is the standard scalar product on $\R^d$ and $|x|=\sqrt{\scal{x}{x}}$ the Euclidean norm. Let $g_1, g_2, \dots$ be independent standard Gaussian random variables. Consider the following Rademacher-Gaussian tail comparison inequality
\begin{equation}\label{eq:R-G}
\p{|\e_1v_1 + \dots + \e_nv_n| \geq t} \leq C\, \p{|g_1v_1+\dots+g_nv_n| \geq t},
\end{equation}
where $v_1, \dots, v_n$ are vectors in $\R^d$. Note that when $d=1$, since sums of independent Gaussians are Gaussian, \eqref{eq:R-G} and \eqref{eq:P1} are equivalent. Pinelis in \cite{P1} first shows that for every even convex function $f$ on $\R$ whose second derivative $f''$ is finite and convex, every $n \geq 1$ and vectors $v_1, \dots, v_n$ in $\R^d$, we have
\begin{equation}\label{eq:pin-f}
\E f(|\e_1v_1 + \dots + \e_nv_n|) \leq \E f(|g_1 v_1 + \dots + g_nv_n|).
\end{equation}
Then he deduces that \eqref{eq:R-G} holds with $C=2e^3/9$ for every $d$, $n$ and vectors $v_1, \dots, v_n$ in $\R^d$ as long as the Gram matrix $A = [\scal{v_k}{v_l}]_{k,l \leq n}$ is an orthogonal projection (equivalently its eigenvalues are $0$ and $1$). In this case $|g_1v_1 + \dots + g_nv_n|^2$ has the chi-square distribution with $\text{rank}(A)$ degrees of freedom ($g_1v_1 + \dots + g_nv_n$ is a standard Gaussian vector on the subspace spanned by the $v_j$), whose log-concavity properties were crucial in the technical parts of Pinelis' proof. We show that the same holds for arbitrary Gram matrices of rank at most $2$.

\begin{theorem}\label{thm:pin}
Inequality \eqref{eq:R-G} holds with $C=3824$ for every $d$, $n$ and vectors $v_1, \dots, v_n$ in $\R^d$ if the subspace they span is $2$-dimensional.
\end{theorem}

Our proof also crucially relies on \eqref{eq:pin-f}. For simplicity of ensuing arguments, but sacrificing values of the constants, to extract a tail bound from \eqref{eq:pin-f}, we adapt ideas from a simpler approach developed in \cite{P5}, rather than the original ones from \cite{P1}. Additionally, it becomes transparent what is needed to remove the restrictions on the matrix $A$ (see remarks in the last section).

\subsection{Stein's property for spherically symmetric random vectors}\label{sec:Stein}

Fix an integer $d \geq 1$ and let $\xi_1, \xi_2, \dots$ be independent random vectors in $\R^d$ uniform on the unit sphere $S^{d-1}$. We are interested in  weighted sums of the $\xi_j$. A fairly general and natural setup is perhaps to let the weights be matrices. We set
\begin{align*}
c_d &= \inf\p{\left|{\textstyle\sum}_{j=1}^n A_j\xi_j\right| \geq \sqrt{\E\left|{\textstyle\sum}_{j=1}^n A_j\xi_j\right|^2}},
\end{align*}
where the infimum is over all $n \geq 1$ and $d \times d$ real matrices $A_1, \dots, A_n$. Let $c_d'$ be this infimum restricted to the matrices which are scalar multiples of the identity matrix. Plainly, $c_1'=c_1$ and $c_d' \geq c_d$. As mentioned in the introduction, Oleszkiewicz showed in \cite{O} that $c_1 \geq \frac{1}{10}$, {\red very recently improved to $c_1 \geq \frac{3}{16}$ by Dvo\v{r}\'ak and Klein in \cite{DK}}. K\"onig and Rudelson have recently showed in \cite{KR} that in general $c_d' \geq \frac{2\sqrt{3}-3}{3+4/d}$, $d \geq 2$, along with better bounds in small dimensions, $c_3' \geq 0.1268$ and $c_4' \geq 0.1407$ (see Proposition 5.1 therein). We extend their result to arbitrary matrix valued coefficients, viz. we provide a lower bound on $c_d$.

\begin{theorem}\label{thm:tom-matrix-coeff}
For every $d \geq 1$, $c_d \geq \frac{7-4\sqrt{3}}{75}$. 
\end{theorem}

Moreover, if we consider the sibling quantity,
\[
C_d = \sup\p{\left|{\textstyle\sum}_{j=1}^n A_j\xi_j\right| > \sqrt{\E\left|{\textstyle\sum}_{j=1}^n A_j\xi_j\right|^2}},
\]
where the supremum is taken again over all $n \geq 1$ and $d \times d$ real matrices $A_1, \dots, A_n$, the proof of Theorem \ref{thm:tom-matrix-coeff} will immediately give a uniform bound on $C_d$ as well.

\begin{corollary}\label{cor:tom-matrix-coeff}
For every $d \geq 1$, $C_d \leq 1 - \frac{7-4\sqrt{3}}{75}$. 
\end{corollary}

\section{Proofs}

\subsection{Auxiliary results}

Both of our results will require at some point to lower bound the probability that a mean zero random variable is positive. This can be done thanks to the following standard Paley-Zygmund type inequality. We include its simple proof for completeness (see also, e.g. \cite{HLNZ} or \cite{Ole}). For results of this type with sharp constants, we refer to \cite{V}.

\begin{lemma}\label{lm:PZ}
Let $Y$ be a mean $0$ random variable such that $\E Y^4 < \infty$. Then
\[
\p{Y \geq 0} \geq 2^{-4/3}\frac{(\E Y^2)^2}{\E Y^4}.
\]
\end{lemma}
\begin{proof}
We can assume that $\p{Y = 0} < 1$. Since $Y$ has mean $0$,
\[
\E|Y| = 2\E Y\1_{Y \geq 0} \leq 2(\E Y^4)^{1/4}\p{Y \geq 0}^{3/4}.
\]
Moreover, by H\"older's inequality, $\E|Y| \geq \frac{(\E Y^2)^{3/2}}{(\E Y^4)^{1/2}}$, so
\[
\p{Y \geq 0} \geq 2^{-4/3}\frac{(\E Y^2)^2}{\E Y^4}.
\]
\end{proof}

\begin{remark}\label{rem:V}
The sharp bound for a non-zero random variable $Y$ with $r  = \frac{\E Y^4}{(\E Y^2)^2}$ reads 
\[
\p{Y > 0} \geq \begin{cases} \frac{1}{2}\left(1 - \sqrt{\frac{r-1}{r+3}}\right), & r \in 1 \leq r < \frac{3}{2}(\sqrt{3}-1), \\
\frac{2\sqrt{3}-3}{r}, & r \geq \frac{3}{2}(\sqrt{3}-1), \end{cases}
\]
see Proposition 2.3 in \cite{V}.
\end{remark}

Since we will need to apply this lemma to sums of independent random variables, it will be convenient to record the following standard computation.

\begin{lemma}\label{lm:4mom}
Let $Y_1, \ldots, Y_n$ be independent mean $0$ random variables such that $\E Y_i^4 \leq L(\E Y_i^2)^2$ for all $1\leq i \leq n$ for some constant $L \geq 1$. Then for $Y = Y_1+\dots+Y_n$,
\[
\E Y^4 \leq \max\{L,3\}(\E Y^2)^2.
\]
\end{lemma}
\begin{proof}
Using independence, $\E Y_i = 0$ and the assumption $\E Y_i^4 \leq L(\E Y_i^2)^2$, we have
\begin{align*}
\E Y^4 = \sum_{i=1}^n \E Y_i^4 + 6\sum_{i < j}\E Y_i^2\E Y_j^2 &\leq \max\{L,3\}\left(\sum_{i=1}^n (\E Y_i^2)^2 + 2\sum_{i < j}\E Y_i^2\E Y_j^2 \right) \\
&= \max\{L,3\}(\E Y^2)^2.
\end{align*}
\end{proof}

In particular, we will also need the following moment comparison involving coordinates of spherically symmetric vectors (which are mildly dependent, nevertheless Lemma \ref{lm:4mom} will be of use here).

\begin{lemma}\label{lm:4mom-sph}
Let $\theta = (\theta_1, \dots, \theta_d)$ be a random vector in $\R^d$ uniform on the unit sphere $S^{d-1}$ and let $a_1, \dots, a_d$ be nonnegative. For $X = \sum_{j=1}^d a_j\theta_j^2$, we have
\[
\E(X-\E X)^4 \leq 15\left(\E|X-\E X|^2\right)^2.
\]
\end{lemma}
\begin{proof}
By homogeneity, we can assume that $\E X = \frac{1}{d}\sum_{j=1}^d a_j= 1$. Then, using $\sum_{j=1}^d \theta_j^2 = 1$,
\[
X - \E X = \sum_{j=1}^d a_j\theta_j^2 - 1 = \sum_{j=1}^d (a_j-1)\theta_j^2 = \sum_{j=1}^d b_j\theta_j^2.
\]
where we put $b_j = a_j - 1$. Note that $\sum_{j=1}^d b_j = 0$. Let $g = (g_1, \dots, g_d)$ be a standard Gaussian random vector in $\R^d$. Then $\frac{g}{|g|}$ has the same distribution as $\theta$ and $\frac{g}{|g|}$ and $|g|$ are independent. Thanks to this independence, for every $p > 0$,
\[
\E\left|\sum_{j=1}^d b_j\theta_j^2\right|^p\cdot \E|g|^{2p} = \E\left|\sum_{j=1}^d b_j\frac{g_j^2}{|g|^2}\right|^p\cdot \E|g|^{2p} = \E\left|\sum_{j=1}^d b_jg_j^2\right|^p = \E\left|\sum_{j=1}^d b_j(g_j^2-1)\right|^p,
\]
where in the last equality we use that $\sum_{j=1}^d b_j = 0$. As a result,
\[
\E|X-\E X|^p = \frac{1}{\E|g|^{2p}}\E\left|\sum_{j=1}^d b_j(g_j^2-1)\right|^p.
\]
Since $\frac{\E(g_j^2-1)^4}{(\E (g_j^2-1)^2)^2}= 15$, from Lemma \ref{lm:4mom}, 
\[
\E\left|\sum_{j=1}^d b_j(g_j^2-1)\right|^4 \leq 15\left(\E\left|\sum_{j=1}^d b_j(g_j^2-1)\right|^2\right)^2
\] 
which together with the obvious bound $\E|g|^8 \geq (\E|g|^4)^2$ yields
\begin{align*}
\E|X-\E X|^4 \leq 15\left(\E|X-\E X|^2\right)^2.
\end{align*}
\end{proof}

\subsection{Proof of Theorem \ref{thm:pin}}

The Gram matrix $A = [\scal{v_k}{v_l}]_{k,l\leq n}$ diagonalises, say $A = U^\top\Lambda U$
for an orthogonal matrix $U$ and a diagonal matrix $\Lambda = \mathrm{diag}(\lambda_1,\ldots,\lambda_n)$ of nonnegative eigenvalues $\lambda_1, \ldots, \lambda_n$. Then
\[
|g_1v_1+\dots+g_nv_n| = \sqrt{g^\top A g} = \sqrt{g^\top U^\top\Lambda Ug},
\]
{\red where $g = (g_1, \dots, g_n)$. Thanks to the rotational invariance of Gaussian measure, $Ug$ has the same distribution as $g$ and as a result, $|g_1v_1+\dots+g_nv_n|$ has the same distribution as}
$\sum_{k=1}^n \lambda_kg_k^2$.

\emph{Case 1: $t \leq \sum_{k=1}^n \lambda_k$.} 
When $t$ is \emph{small}, there is nothing to do because the right hand side is at least $1$ if we choose $C$ large enough. More precisely, we have
\begin{equation}\label{eq:small-t}
\p{\sum_{k=1}^n \lambda_kg_k^2 > \sum_{k=1}^n \lambda_k} \geq \frac{1}{15\cdot 2^{4/3}}.
\end{equation}
This follows from Lemmas \ref{lm:PZ} and \ref{lm:4mom} applied to $Y_k = \lambda_k(g_k^2-1)$ for which we have $\frac{\E Y_k^4}{(\E Y_k^2)^2} = 15$ (the constant $\frac{1}{15\cdot 2^{4/3}}$ can be improved to $\frac{2\sqrt{3}-3}{15}$, see Proposition 3.5 in \cite{V}).

\emph{Case 2: $t \geq \sum_{k=1}^n \lambda_k$.}
If $A$ has rank at most $2$, then at most two of the $\lambda_k$ are nonzero. If only one is nonzero ($A$ has rank $1$), the theorem reduces to Pinelis' result. Suppose that $A$ has rank $2$. By homogeneity, we can assume that the eigenvalues $\lambda_k$ are $1, \lambda^{-1}, 0, \ldots, 0$ for some $\lambda \geq 1$. By Markov's inequality combined with Pinelis' result \eqref{eq:pin-f}, we obtain
\[
\p{|\e_1v_1 + \dots + \e_nv_n| > t} = \p{\sqrt{\e^\top A \e } > t} \leq  \frac{\E f(\sqrt{\e^\top A \e})}{f(t)} \leq \frac{\E f(\sqrt{g^\top A g})}{f(t)}
\]
for every $t > 0$ and every function $f(x)$ of the form $f(x) = (x-u)_+^3$ with $0 < u < t$. The proof is finished with the following lemma applied to $X = \sqrt{g^\top A g}$.

\begin{lemma}\label{lm:f-to-tail}
Let $X = \sqrt{g_1^2 + \lambda^{-1}g_2^2}$ with $\lambda \geq 1$ and $g_1, g_2$  independent standard Gaussian random variables. For every $t > 1$ there is $0 < u < t$ such that
\[
\frac{\E (X-u)_+^3}{(t-u)_+^3} \leq C_0\p{X > t}
\]
with a universal constant $C_0 > 0$. Moreover, we can take $C_0 = 3824$.
\end{lemma}
\begin{proof}
Let $f_\lambda(t)$ be the density of $X$,
\[
f_\lambda(t) = \lambda^{1/2} t \exp\left(-\frac{\lambda+1}{4}t^2\right)I_0\left(\frac{\lambda-1}{4}t^2\right)\1_{t > 0},
\]
where $I_0(s) = \frac{1}{\pi}\int_0^\pi \exp(s\cos \theta)\dd\theta$ stands for the modified Bessel function of the first kind. We need two technical claims about $f_\lambda$ (we defer their proofs).

\textbf{Claim 1.} For every $\lambda \geq 1$, $f_\lambda$ is log-concave on $(\frac{3}{4},\infty)$.

\textbf{Claim 2.} For every $\lambda \geq 1$, $f_\lambda(1) > \sqrt{\frac{2}{\pi e}}$.

By Claim 1 and the Pr\'ekopa-Leindler inequality, the tail function $h(t) = \p{X > t}$ is also log-concave on $(t_0,\infty)$, $t_0 = \frac{3}{4}$ (see, e.g. Proposition 5.4 in \cite{GNT}). Fix $0 < u < t$ and write
\[
\E (X-u)_+^3 = \int_u^\infty 3(x-u)^2h(x) \dd x.
\]
If we choose $u > t_0$, using the supporting tangent line of the convex function $-\log h$ at $x=t$, we have
\begin{equation}\label{eq:tangent}
h(x) \leq h(t)e^{-a(x-t)}, \qquad x > u,
\end{equation}
where $a = (-\log h)'(t) = -\frac{h'(t)}{h(t)} > 0$ (as $h$ is strictly decreasing). Thus
\[
\E (X-u)_+^3 \leq 3h(t)\int_u^\infty (x-u)^2e^{-a(x-t)} \dd x = 6h(t) \frac{e^{a(t-u)}}{a^3}.
\]
Setting $u = t - \frac{c}{a}$ with $c = (1-t_0)\sqrt{\frac{2}{\pi e}}$ yields
\[
\E (X-u)_+^3 \leq 6h(t) \frac{e^{a(t-u)}}{a^3} = \frac{6e^c}{c^3}(t-u)^3h(t).
\]
It remains to check that for this choice of $u$, we indeed have $u > t_0$, as required earlier. Since $a$, as a function of $t$, is nondecreasing (as $h$ is log-concave), for every $t > 1$, we have
\[
t - \frac{c}{a} > 1 - \frac{c}{-\frac{h'(1)}{h(1)}} = 1 - c\frac{h(1)}{f_\lambda(1)} > 1 - c\frac{1}{\sqrt{2/(\pi e)}} = t_0,
\]
where in the last inequality we use that trivially $h(1) < 1$ and $f_\lambda(1) > \sqrt{\frac{2}{\pi e}}$, by Claim 2. Thus the lemma holds with $C_0 = \frac{6e^c}{c^3} < 3824$.
\end{proof}

\begin{proof}[Proof of Claim 1.]
Letting $a = \frac{\lambda+1}{2}$ and $b = \frac{\lambda-1}{2}$, we write
\[
f_\lambda(t) = \lambda^{1/2}te^{-at^2/2}I_0(bt^2/2),
\]
differentiate (using $I_0'(x) = I_1(x)$ and $I_1'(x) = I_0(x) - \frac{1}{x}I_1(x)$), to obtain
\begin{align*}
\lambda^{-1}e^{{\red at^2}}\Big((f_\lambda')^2(t) - f_\lambda''(t)f_\lambda(t)\Big) &= (1+at^2-(bt^2)^2)I_0^2 + bt^2I_0I_1 + (bt^2)^2I_1^2 \\
&= I_0^2\left(\left(2uR(u)+\frac{1}{2}\right)^2 - \left(2u-\frac{1}{2}\right)^2 + 1 + t^2\right)
\end{align*}
where $R = \frac{I_1}{I_0}$ and all the functions on the right hand side are evaluated at $u = bt^2/2$. Thus to show that $(f_\lambda')^2(t) - f_\lambda''(t)f_\lambda(t) > 0$ for every $\lambda \geq 1$ and $t > \frac{3}{4}$, it suffices to show that for every $u > 0$, we have
\begin{equation}\label{eq:R}
\left(2uR(u)+\frac{1}{2}\right)^2 - \left(2u-\frac{1}{2}\right)^2 + 1 + \left(\frac{3}{4}\right)^2 > 0.
\end{equation}
By results of N\aa sell (see Theorem 3 in \cite{N}),
\[
R(u) \geq L_{0,5,1}(u), \qquad u > 0,
\]
with
\[
L_{0,5,1}(u) = \frac{u (120960 + 60480 u + 25200 u^2 + 7140 u^3 + 1455 u^4 + 204 u^5 + 
   16 u^6)}{241920 + 120960 u + 80640 u^2 + 29400 u^3 + 7950 u^4 + 
 1563 u^5 + 212 u^6 + 16 u^7}.
\]
Thus to show \eqref{eq:R}, it suffices to show the same inequality with $R(u)$ replaced by $L_{0,5,1}(u)$. The left hand side then becomes $\frac{P(u)}{Q(u)}$ with
\begin{align*}
P(u) =& 1\, 463\, 132\, 160\, 000 + 3\, 335\, 941\, 324\, 800 u + 404\, 799\, 897\, 600 u^2 \\
&- 
   249\, 138\, 892\, 800 u^3 - 239\, 747\, 558\, 400 u^4 - 55\, 539\, 993\, 600 u^5 \\
   &+ 1\, 473\, 272\, 640 u^6 + 4\,994\, 831\, 520 u^7 + 1\, 686\, 522\, 420 u^8 + 309\, 775\, 380 u^9 \\
   &+ 28\, 100\, 385 u^{10} - 1\, 681\, 032 u^{11} + 768\, 112 u^{12} + 57\, 984 u^{13} + 
   2\, 304 u^{14}
\end{align*}
and
\begin{align*}
Q(u) &= 16 (241920 + 120960 u + 80640 u^2 + 29400 u^3 + 7950 u^4 + 1563 u^5 + 
   212 u^6 + 16 u^7)^2.
\end{align*}
It suffices to show that the polynomial $P(u)$ is positive for $u > 0$. Write it as $P(u) = \sum_{k=0}^{14} a_ku^k$. For $u \in (0,2)$, plainly  
\begin{align*}
a_0+(a_5+10^{10})u^5 &> a_0 + (a_5+10^{10})\cdot 2^5 > 0, \\
a_2u^2 - 10^{10}u^5 &> u^2(a_2 - 10^{10}\cdot 2^3) > 0, \\
a_1u + a_3u^3 + a_4u^4 &> u(a_1 + a_3\cdot 2^2 + a_4\cdot 2^3) > 0, \\
a_{10}u^{10}+a_{11}u^{11} &> u^{10}(a_{10}+2a_{11}) > 0, \\
a_ku^k &> 0, \qquad k = 6, 7, 8, 9, 12, 13, 14.
\end{align*}
Adding these together shows that $P(u) > 0$, $u \in (0,2)$. Finally, writing $P(u+2) = \sum_{k=0}^{14} b_ku^k$, we get that $b_k > 0$ for all $k \geq 5$, so $\sum_{k=5}^{14} b_ku^k > 0$ for all $u > 0$ and using standard formulae for the discriminant of the quartic part $\sum_{k=0}^4 b_ku^4$, we check that it has no real roots, so it is positive everywhere (as being positive at $u=0$), hence $P(u) > 0$ also for all $u > 2$.
\end{proof}

\begin{proof}[Proof of Claim 2.]
We have $f_\lambda(1) = \sqrt{\lambda}e^{-\frac{\lambda+1}{4}}I_0(\frac{\lambda-1}{4})$, so letting $u = \frac{\lambda-1}{4}$, we want to show that for every $u > 0$,
\[
\sqrt{4u+1}e^{-u-1/2}I_0(u) > \sqrt{\frac{2}{\pi e}}.
\]
Equivalently,
\[
\int_0^{\pi} e^{u(\cos \theta - 1)} \dd \theta > \sqrt{\frac{2\pi}{4u+1}}, \qquad u > 0.
\]
Using $\cos\theta \geq 1 - \theta^2/2$ and changing the variables $s = \theta\sqrt{u}$, it suffices to show that
\[
\int_0^{\pi\sqrt{u}} e^{-s^2/2} \dd s - \sqrt{\frac{2\pi u}{4u+1}} > 0, \qquad u > 0.
\]
Call the left hand side $\psi(u)$. We have, $\psi(0) = 0$ and $\psi(\infty) = 0$, so it is enough to show that $\psi'$ is first positive and then negative. We have,
\[
\psi'(u) = \sqrt{\frac{\pi}{2u}}\left(\sqrt{\frac{\pi}{2}}e^{-\pi^2u/2}-(4u+1)^{-3/2}\right).
\]
The sign of $\psi'$ is thus the same as of $\log\sqrt{\frac{\pi}{2}}-\frac{\pi^2}{2}u+\frac{3}{2}\log(4u+1)$ which is plainly strictly concave, is positive at $u=0$ and tends to $-\infty$ as $u \to \infty$, therefore is first positive and then negative.
\end{proof}

\subsection{Proof of Theorem \ref{thm:tom-matrix-coeff}}

Our goal is to show that for every $n \geq 1$ and $d \times d$ real matrices $A_1, \dots, A_n$, we have
\begin{equation}\label{eq:matrix-goal}
\p{\left|\sum_{j=1}^n A_j\xi_j\right|^2 \geq \E\left|\sum_{j=1}^n A_j\xi_j\right|^2} \geq \frac{7-4\sqrt{3}}{{\red 75}}.
\end{equation}
A natural approach would be to use Lemma \ref{lm:PZ}, however comparing the second and fourth moments of $Y=\left|\sum_{j=1}^n A_j\xi_j\right|^2 - \E\left|\sum_{j=1}^n A_j\xi_j\right|^2$ does not seem to be approachable through a direct computation (in the case when each $A_j$ is a scalar multiple of the identity matrix, $Y$ becomes a quadratic form in $\scal{\xi_j}{\xi_k}$ which is managable, as done in \cite{KR}). Instead, we shall first exploit the symmetry of the $\xi_j$. Let $\e_1, \e_2, \dots$ be independent Rademacher random variables, also independent of the sequence $\xi_1, \xi_2, \dots$. Note that the sequences $(\xi_j)$ and $(\e_j\xi_j)$ have the same distribution. Set 
\[
\mu = \E\left|\sum_{j=1}^n A_j\xi_j\right|^2 = \sum_{j=1}^n \E|A_j\xi_j|^2.
\]
We have,
\begin{align*}
\mathbb{P}\left(\left|\sum_{j=1}^n A_j\xi_j\right|^2 \geq \mu\right)&=\mathbb{P}_{\e,\xi}\left(\left|\sum_{j=1}^n \e_jA_j\xi_j\right|^2 \geq \mu\right) \\
&\geq \mathbb{P}_{\e,\xi}\left(\left|\sum_{j=1}^n \e_jA_j\xi_j\right|^2 \geq \sum_{j=1}^n |A_j\xi_j|^2, \ \sum_{j=1}^n |A_j\xi_j|^2 \geq \mu\right) \\
&=\E_\xi\left[ \mathbb{P}_{\e}\left(\left|\sum_{j=1}^n \e_jA_j\xi_j\right|^2 \geq \sum_{j=1}^n |A_j\xi_j|^2\right)\1_{\left\{\sum_{j=1}^n |A_j\xi_j|^2 \geq \mu\right\}}\right].
\end{align*}
We know from (3.8) in Corollary 3.4 from \cite{V} that for arbitrary vectors $v_1, \dots, v_n$ in $\R^d$, we have
\begin{equation}\label{eq:radem}
\mathbb{P}_{\e}\left(\left|\sum_{j=1}^n \e_jv_j\right|^2 \geq \sum_{j=1}^n |v_j|^2\right) \geq \frac{2\sqrt{3}-3}{15}.
\end{equation}
Thus
\[
\mathbb{P}\left(\left|\sum_{j=1}^n A_j\xi_j\right|^2 \geq \mu\right)\geq \frac{2\sqrt{3}-3}{15}\p{\sum_{j=1}^n |A_j\xi_j|^2 \geq \mu}.
\]
Finally, to lower bound the probability on the right hand side, we first remark that here, without loss of generality, we can assume that the matrices $A_j$ are diagonal. This is because invoking the singular value decomposition, $A_j = V_j\Lambda_jU_j$ with $U_j, V_j$ orthogonal and $\Lambda_j$ diagonal $d \times d$ matrices. Since $|A_j\xi_j| = |U_j\Lambda_jV_j\xi_j| = |\Lambda_jV_j\xi_j|$, by rotational symmetry, $|A_j\xi_j|$ has the same distribution as $|\Lambda_j\xi_j|$. In the case when the $A_j$ are diagonal, from Lemma \ref{lm:4mom-sph}, \[
\E(|A_j\xi_j|^2 - \E|A_j\xi_j|^2)^4 \leq 15\big(\E(|A_j\xi_j|^2 - \E|A_j\xi_j|^2\big)^2, 
\]
so Lemma \ref{lm:4mom} combined with Remark \ref{rem:V} yields
\[
\p{\sum_{j=1}^n |A_j\xi_j|^2 \geq \mu} \geq \frac{2\sqrt{3}-3}{15},
\]
which inserted into the previous bound finishes the proof.\hfill$\square$

\subsection{Proof of Corollary \ref{cor:tom-matrix-coeff}}
We repeat verbatim the proof of Theorem \ref{thm:tom-matrix-coeff} with each ``$\geq$'' replaced by ``$\leq$'' in all of the events considered:  {\red for inequality \eqref{eq:radem} this is justified again by Corollary 3.4 from \cite{V} (with (3.7) used instead of (3.8)) and in the very last step Remark \ref{rem:V} is applied to $-Y$ instead of $Y$.} This way we obtain that
\[
\p{\left|\sum_{j=1}^n A_j\xi_j\right|^2 \leq \E\left|\sum_{j=1}^n A_j\xi_j\right|^2} \geq \frac{7-4\sqrt{3}}{75},
\]
equivalently, $C_d \leq 1 - \frac{7-4\sqrt{3}}{75}$.\hfill$\square$

\section{Further remarks}

\subsection{Constant in Theorem \ref{thm:pin}}
Instead of the simple convexity argument \eqref{eq:tangent} of Lemma \ref{lm:f-to-tail}, adapting the proof of Theorem 2.4 from \cite{P1}, after somewhat lengthy and nontrivial computations, Lemma \ref{lm:f-to-tail} can be established with $C_0 = \frac{3e^2}{4}$. As a result, the value of the constant $C$ in Theorem \ref{thm:pin} can be improved to $\frac{3e^2}{4}$.

\subsection{Extensions of \eqref{eq:R-G}} 
We known that \eqref{eq:R-G} holds with a universal constant when the Gram matrix of the vectors $v_j$ has eigenvalues in the set $\{0,1\}$ (see \cite{P1}), or when the vectors $v_j$ all lie in a $2$-dimensional subspace (Theorem \ref{thm:pin}). We conjecture that \eqref{eq:R-G} continues to hold with a universal constant for every $d$ and every $n$ vectors in $\R^d$. To establish that, it would be enough to have analogues of Claims 1 and 2, essentially to the effect that $f_\lambda$ is log-concave on $(s, \infty)$ and $f_\lambda(s) > c_0$ for a universal constant $c_0$, where now  $f_\lambda$ is the density of $(\sum_{j=1}^k \lambda_j g_j^2)^{1/2}$ and $s = (\sum_{j=1}^k \lambda_j)^{1/2}$, given a positive sequence $\lambda = (\lambda_j)_{j=1}^k$.

We also know that a multidimensional analogue of \eqref{eq:R-G} in the spirit of Section \ref{sec:Stein} holds for scalar coefficients (see \cite{NT} and \cite{P6} for two different approaches). It would perhaps be interesting to investigate a generalisation to matricial coefficients. 

We remark that the main result of \cite{NT} and \cite{P6} specialised to dimension 2 provides a different complex analogue of \eqref{eq:P1} than the main result of this paper, namely that there is a universal constant $C> 0 $ such that \eqref{eq:R-G} holds for every $n \geq 1$, $v_1, \dots, v_n \in \C$ with the $\epsilon_j$ being independent uniform on the unit circle $\{z \in \C, |z| = 1\}$ and the $g_j$ independent \emph{standard} Gaussian in $\C$ (i.e. with density $\frac{1}{\pi}e^{-|z|^2}$, $z \in \C$).

Finally, the Euclidean norm $|\cdot|$ in \eqref{eq:R-G} cannot be replaced with an arbitrary norm. For instance, for the $\ell_1$ norm $\|\cdot\|_1$ and the standard basis, we have  $\|\sum_{j=1}^d \e_je_j\|_1 = d$, whereas $\|\sum_{j=1}^d \e_jg_j\|_1 = \sum_{j=1}^d |g_j|$ which concentrates around its expectation which is $\sqrt{\frac{2}{\pi}}d$ and in fact $\p{\|\sum_{j=1}^d \e_jg_j\|_1 \geq d} \leq \exp(-c d)$ for a universal constant $c$.

\subsection{Typical probabilities in high dimensions}
For the constant $c_d'$ and $c_d$ defined in Section \ref{sec:Stein}, in high-dimensions, that is as $d \to \infty$, we conjecture that $c_d' = \frac{1}{2}-o(1)$ and $c_d = \left(\sqrt{\frac{2}{\pi}}\int_{1}^\infty e^{-u^2/2}\dd u\right) -o(1)$ (furnished by the examples of $A_1 = \dots = A_n = \frac{1}{\sqrt{n}}\text{Id}$ and $A_1 = \dots = A_n = \frac{1}{\sqrt{n}}\text{diag}(1,0,\dots,0)$, respectively, see also Remark 5.2(b) in \cite{KR}).

\end{document}